\documentclass[11pt]{article}

\usepackage[a4paper,margin=1.1in]{geometry}
\usepackage{setspace}
\usepackage[T1]{fontenc}
\usepackage{lmodern}
\usepackage{microtype}

\usepackage{amsmath,amssymb,amsthm}
\usepackage{mathtools}
\usepackage{mathrsfs}
\usepackage{pgf,tikz}
\usetikzlibrary{automata}
\usetikzlibrary{arrows}
\usepackage[numbers]{natbib}

\usepackage[hidelinks]{hyperref}

\usepackage{fancyhdr}
\theoremstyle{plain}
\newtheorem{theorem}{Theorem}[section]

\newtheorem{proposition}[theorem]{Proposition}

\theoremstyle{definition}
\newtheorem{definition}[theorem]{Definition}
\newtheorem{example}[theorem]{Example}

\theoremstyle{remark}
\newtheorem{remark}[theorem]{Remark}

\title{\textbf{On $\mathscr{M}$-arrangements of conics and lines with ordinary singularities}}
\author{
  Marek Janasz and Piotr Pokora
}
\date{\today}

\begin{document}
\maketitle

\begin{abstract}
In this paper we study combinatorial aspects of reduced plane curves known as $\mathscr{M}$-curves. This notion is a natural generalization of maximizing plane curves, which are well-known in the theory of algebraic surfaces. We focus on $\mathscr{M}$-arrangements of conics and lines with ordinary singularities of multiplicity at most four. We provide numerical constraints on their existence, especially in terms of weak combinatorics, and then study in detail the case of arrangements consisting of one conic and lines. We also construct a new example with one conic and eleven lines, prove boundedness results for real arrangements of this type, and record a regularity consequence for the associated Milnor algebra and module of Jacobian syzygies.
\end{abstract}

\section{Introduction}\label{s:introduction}

In the present paper we study a relatively new class of plane curves, called
\(\mathscr{M}\)-curves. These curves were recently introduced in \cite{JanLes} as a
natural generalization of the classical maximizing curves introduced by Persson
\cite{Persson}. The theory of maximizing curves of even degree is well-established,
and several interesting constructions of such curves are known. Maximizing plane
curves admit only \({\rm ADE}\)-singularities and are free curves whose Jacobian
schemes have the largest possible degree. These conditions are rather restrictive,
which makes explicit constructions difficult. Recently, Dimca and Pokora introduced
maximizing curves of odd degree in \cite{DimPok}. Such curves seem to be extremely
rare; at present, only a few examples are known and only in low degree. This motivates the
study of a slightly larger class of reduced plane curves which still shares the key
numerical and homological properties of maximizing curves. This is the point of view
behind the notion of \(\mathscr{M}\)-curves.

Let us recall the definition in the form needed below; see
\cite[Definitions 4.1 and 4.7]{JanLes}.

\begin{definition}
Let \(C=\{f=0\}\subset \mathbb{P}^{2}_{\mathbb C}\) be a reduced curve of degree
\(d\geq 5\) admitting only \({\rm ADE}\)-singularities and simple elliptic singularities,
that is, ordinary quadruple points and \(J_{2,0}\)-singularities. Denote by \(J_f\) the
Jacobian ideal generated by the partial derivatives of \(f\). We say that \(C\) is an
\(\mathscr{M}\)-curve if
\[
{\rm deg}\,J_f=
\begin{cases}
3m^{2}-3m+3 & \text{if } d=2m\geq 6,\\
3m^{2}+1 & \text{if } d=2m+1\geq 5.
\end{cases}
\]
\end{definition}

The notion of an \(\mathscr{M}\)-curve is meaningful only when the curve has at least
one simple elliptic singularity, and this will be our standing assumption throughout
the paper.

Our goal is to investigate combinatorial properties of \(\mathscr{M}\)-curves in the
special but natural setting of conic-line arrangements with only ordinary double,
triple and quadruple points. The first steps toward understanding
\(\mathscr{M}\)-curves among line arrangements were taken in \cite{JanLes05}. For
instance, every real \(\mathscr{M}\)-arrangement of lines with only double, triple and
quadruple points is simplicial; see \cite{JanLes05}. This non-trivial and somewhat
unexpected fact shows that \(\mathscr{M}\)-curves form a restrictive and delicate class.

Let us also mention two related viewpoints in which Jacobian syzygies play a central role.
The first one comes from the theory of algebraic curves invariant under polynomial
differential equations. If \(C=\{f=0\}\subset \mathbb P^2\) is invariant under a polynomial
vector field, then the associated logarithmic derivation gives a non-trivial relation among
the partial derivatives of \(f\). Thus invariant arrangements naturally lead to Jacobian
syzygies of controlled degree, which are precisely the objects entering the study of
\({\rm mdr}(f)\), freeness and related properties of plane curves. In this direction,
de Moura Canaan and Coutinho \cite{CC} classified arrangements of ten real projective
lines invariant under polynomial differential equations of degree four.

A second related direction concerns conic-line arrangements with very low minimal degree
of Jacobian syzygies. Beorchia and Mir\'o-Roig in \cite{BM} classified arrangements with a
linear Jacobian syzygy, that is, with \({\rm mdr}(f)=1\), attaining the du Plessis--Wall
bounds for the global Tjurina number. Their configurations arise from special pencils of
conics, such as hyperosculating or bitangent pencils, and typically have non-ordinary
singularities, for instance tacnodes or \(A_7\)-singularities. These works are therefore
complementary to the present paper: our focus is on conic-line arrangements with only
ordinary double, triple and quadruple points satisfying the numerical conditions defining
\(\mathscr{M}\)-curves, and on the resulting combinatorial and freeness consequences,
especially after deleting the distinguished conic.

It is also worth emphasizing that arrangements of conics and lines arise naturally in
constructions of algebraic surfaces with extremal properties; see, for instance,
\cite{NP26, Pard, Pelka}. This gives another motivation for studying this class of curves.

We now summarize the main results of the paper. After recalling the necessary
preliminaries on free plane curves and \(\mathscr{M}\)-curves, we prove a numerical
criterion for the weak combinatorics of \(\mathscr{M}\)-arrangements of conics and
lines with ordinary singularities of multiplicity at most four; see Theorem
\ref{char}. We then specialize to arrangements consisting of one smooth conic and
\(d\) lines. In this case we derive explicit constraints on the numbers of double,
triple and quadruple points, and we study the number of singular points lying on the
distinguished conic. These restrictions lead to a precise description of the possible
combinatorial Poincar\'e polynomials of the line arrangements obtained after deleting
the conic; see Theorem \ref{poil}.

Next, we construct an explicit \(\mathscr{M}\)-arrangement consisting of one conic and
eleven lines, apparently new, whose deletion is a free arrangement of eleven lines. We
also prove that real \(\mathscr{M}\)-arrangements consisting of one conic and lines are
bounded in degree. Finally, we record a consequence for the Castelnuovo--Mumford
regularity of the Milnor algebra and the module of Jacobian syzygies associated with
an \(\mathscr{M}\)-curve.

\section{Preliminaries}
In this paper we adopt the notation from \cite{Dimca}. Let $S = \mathbb{C}[x,y,z]$ be the coordinate ring of $\mathbb{P}^{2}_{\mathbb{C}}$. For a homogeneous polynomial $f \in S$ we denote by $J_{f}$ the Jacobian ideal associated with $f$, i.e., the ideal of the form $J_{f} = \langle \partial_{x}\, f, \partial_{y} \, f, \partial_{z} \, f \rangle$. We define the Milnor algebra associated to $f$ as $M(f) = S/ J_{f}$. 

We will need an important invariant that is defined in the language of the syzygies of $J_{f}$.
\begin{definition}
Consider the graded $S$-module of Jacobian syzygies of $f$, namely $$AR(f)=\{(a,b,c)\in S^3 : a\partial_{x} \, f + b \partial_{y} \, f + c \partial_{z} \, f = 0 \}.$$
The minimal degree of non-trivial Jacobian relations for $f$ is defined
$${\rm mdr}(f):=\min_{r\geq 0}\{AR(f)_r\neq 0\}.$$ 
\end{definition}
\begin{remark}
If $C = \{f=0\} \subset \mathbb{P}^{2}_{\mathbb{C}}$ is a reduced curve then we write ${\rm mdr}(f)$ or ${\rm mdr}(C)$ interchangeably.
\end{remark}
Let us now formally define the freeness property for reduced plane curves.
\begin{definition}
A reduced curve $C \subset \mathbb{P}^{2}_{\mathbb{C}}$ of degree $d$ is free if the Jacobian ideal $J_{f}$ is saturated with respect to $\mathfrak{m} = \langle x,y,z\rangle$. Moreover, if $C$ is free, then the pair $(d_{1}, d_{2}) = ({\rm mdr}(f), d - 1 - {\rm mdr}(f))$ is called the exponents of $C$, i.e., the degrees of generators of ${\rm AR}(f)$.
\end{definition}
\begin{remark}
In the scenario where $C = \{f=0\}\subset \mathbb{P}^{2}_{\mathbb{C}}$ is free, generators of ${\rm AR}(f) $ are not uniquely determined, contrary to the degree of these generators. 
\end{remark}
\begin{remark}
It is worth recalling here that $\mathscr{M}$-curves are free, see \cite{JanLes} for details.
\end{remark}

In order to check whether a given plane curve $C\subset \mathbb{P}^{2}_{\mathbb{C}}$ is free we can use the following criterion that comes from~\cite{duP}. Let us recall that for a reduced curve $C = \{f=0\}\subset \mathbb{P}^{2}_{\mathbb{C}}$ one has
$${\rm deg} \, J_{f} = \tau(C) = \sum_{p \in {\rm Sing}(C)} \tau_{p}$$
with $\tau_{p}$ being the Tjurina number of a singular point $p \in C$.
\begin{theorem}[du-Plessis -- Wall]
\label{dup}
Let $C = \{f=0\}\subset \mathbb{P}^{2}_{\mathbb{C}}$ be a reduced curve of degree $d$ and let $r = {\rm mdr}(f)$. Let us denote the total Tjurina number of $C$ by $\tau(C)$.
Then the following two cases hold.
\begin{enumerate}
\item[a)] If $r < d/2$, then $\tau(C) \leq \tau_{max}(d,r)= (d-1)(d-r-1)+r^2$ and the equality holds if and only if the curve $C$ is free.
\item[b)] If $d/2 \leq r \leq d-1$, then
$\tau(C) \leq \tau_{max}(d,r)$,
where, in this case, we set
$$\tau_{max}(d,r)=(d-1)(d-r-1)+r^2- \binom{2r-d+2}{2}.$$
\end{enumerate}
\end{theorem}
From now on we are working with conic-line arrangements $\mathcal{CL}\subset \mathbb{P}^{2}_{\mathbb{C}}$ that admit only ordinary quasi-homogeneous singularities. In particular, all ordinary singularities of multiplicity at most $4$ are quasi-homogeneous \cite[Exercise 7.31]{RCS}.
\begin{remark}
Since we will only work with curves admitting ordinary quasi-homogeneous singularities, the Tjurina numbers of singular points are equal to the Milnor numbers \cite{Reiffen}. Hence, we have the following chain of equalities:
$$\tau(C) = \sum_{p \in {\rm Sing}(C)} \tau_{p} = \sum_{p \in {\rm Sing}(C)} \mu_{p} = \sum_{r\geq 2}(r-1)^{2}n_{r},$$
where $n_{r}$ denotes the number of $r$-fold intersection points, i.e., points in the plane where exactly $r$ curves meet, and $\mu_p$ denotes the Milnor number of a singular point $p \in C$.
\end{remark}
\begin{remark}
In the light of the above remark, if $C=\{f=0\}\subset \mathbb{P}^{2}_{\mathbb{C}}$ is a reduced curve of degree $d$ admitting only ordinary quasi-homogeneous singularities, then $C$ is free if and only if $r={\rm mdr}(f) \leq (d-1)/2$ and 
\begin{equation}
\label{DPW}
(d-1)(d-r-1)+r^2 = \sum_{r\geq 2}(r-1)^{2}n_{r}.
\end{equation}
\end{remark}
In this paper, we use the concept of weak combinatorics, which can be defined as data attached to a given conic-line arrangement, represented by a vector of the form:  
$$C(d,k;n_2, \ldots, n_t).$$ Here, $d$ denotes the number of lines, $k$ denotes the number of smooth conics, $n_r$ denotes, as above, the number of $r$-fold ordinary intersections, and $t$ is equal to the maximal multiplicity among singular points $p \in {\rm Sing}(\mathcal{CL})$.

\section[Constraints on weak combinatorics of M-arrangements of conics and lines]{Constraints on weak combinatorics of $\mathscr{M}$-arrangements of conics and lines}
We present our main combinatorial description of $\mathscr{M}$-arrangements of conics and lines with ordinary singularities.
\begin{theorem}
\label{char}
Let $\mathcal{CL} \subset \mathbb{P}^{2}_{\mathbb{C}}$ be an arrangement of $k\geq 1$ conics and $d\geq 3$ lines with ordinary singularities of multiplicity at most $4$. Assume that $\mathcal{CL}$ is an $\mathscr{M}$-arrangement.
\begin{itemize}
    \item[(i)] If $d=2\ell$ for some $\ell\geq 2$, then 
    $$n_{2}+2n_{3}+3n_{4} = (k+\ell)^{2}+\ell-k-3.$$
    \item[(ii)] If $d=2\ell+1$ for some $\ell \geq 1$, then
    $$n_{2}+2n_{3}+3n_{4} = (k+\ell)^{2}+2\ell - 1.$$
\end{itemize}
\end{theorem}
\begin{proof}
First, let us consider the case $d=2\ell \geq 4$. Recall that by \cite[Theorem 4.3]{JanLes} we have
$$\tau(\mathcal{CL})=n_{2}+4n_{3}+9n_{4}=3(k+\ell)^2 - 3(k+\ell) +3.$$ 
After applying B\'ezout's Theorem to $\mathcal{CL}$, we get the following na\"ive combinatorial count:
\begin{equation}
\label{naive1}
4\cdot \binom{k}{2} + 4k\ell + \binom{2\ell}{2} = n_{2} + 3n_{3} + 6n_{4},
\end{equation}
Let us rewrite this count as follows:
$$4(k^{2}-k) +8k\ell + 2\ell(2\ell-1)= 2n_{2}+6n_{3}+12n_{4} = \tau(\mathcal{CL}) + n_{2}+2n_{3}+3n_{4}.$$
We get
\begin{multline*}
4(k^{2}-k) +8k\ell + 2\ell(2\ell-1) - \tau(\mathcal{CL}) = \\ 4(k^{2}-k) +8k\ell + 2\ell(2\ell-1) - 3(k+\ell)^2 + 3(k+\ell) -3  = (k+\ell)^{2}+\ell - k - 3,
\end{multline*}
and hence
$$(k+\ell)^{2}+\ell - k - 3= n_{2}+2n_{3}+3n_{4},$$
which completes the proof of the first case.

\noindent
Now we pass to the second case, where $d=2\ell + 1 \geq 3$. We recall that by \cite[Theorem 4.8]{JanLes} one has
$$\tau(\mathcal{CL})=n_{2}+4n_{3}+9n_{4}=3(k+\ell)^2 + 1.$$ 
We use again B\'ezout's Theorem applied to $\mathcal{CL}$ obtaining the na\"ive combinatorial count:
\begin{equation}
4\cdot \binom{k}{2} + 2k(2\ell+1) + \binom{2\ell+1}{2} = n_{2} + 3n_{3} + 6n_{4},
\end{equation}
and we get
$$4k^{2}+8k\ell+4\ell^{2}+2\ell = \tau(\mathcal{CL}) + n_{2}+2n_{3}+3n_{4}.$$
After some manipulations we finally arrive at
$$(k+\ell)^{2}+2\ell - 1 = n_{2}+2n_{3}+3n_{4},$$
which completes the proof.
\end{proof}
\begin{example}
Consider the following combinatorics $C(d,k;n_{2},n_{3},n_{4}) = (7,1;5,4,3)$. Using Theorem \ref{char} $(ii)$ with $\ell=3$, we see that
$$22= 5 + 4\cdot2 + 3\cdot 3 = n_{2} +2n_{3}+3n_{4} \neq (k+\ell)^2 + 2\ell - 1 = 4^2+6 - 1 = 21,$$
hence this combinatorics cannot lead to an $\mathscr{M}$-curve.
\end{example}
\begin{example}
Let us consider $\mathcal{CL}\subset \mathbb{P}^{2}_{\mathbb{C}}$ given by the following defining polynomial
$$f(x,y,z)=xy(x+y-z)(x^2 - xz + 5y^2 - 5yz)(x^2+2y^2 - xz - 2yz).$$
The arrangement $\mathcal{CL}$ has weak combinatorics $C(3,2;1,0,3)$ and it is known to be free with exponents $(d_{1},d_{2})=(2,4)$, see \cite[Example 2.1]{ST}, and hence $\mathcal{CL}$ is an $\mathscr{M}$-arrangement. Since $\ell=1$ we can check that
$$10 = n_{2}+2n_{3}+3n_{4} = (k+\ell)^{2}+2\ell-1 = 3^2+2-1 = 10,$$
which justifies Theorem \ref{char} $(ii)$.
\end{example}
Now we would like to study a special case, namely we consider $\mathscr{M}$-arrangements consisting of $d$ lines and just one conic. Our result provides a complete set of constraints that allow us to find all weak combinatorics of such arrangements.
\begin{proposition}
\label{kon}
Let $\mathcal{CL} \subset \mathbb{P}^{2}_\mathbb{C}$ be an $\mathscr{M}$-arrangement of $d\geq 3$ lines and one conic with only ordinary singularities of multiplicity at most $4$.
\begin{itemize}
    \item[(i)] If $d=2\ell\geq 4$, then $n_{2}+n_{3}=3\ell-6$ and $n_{3}+3n_{4} = \ell^{2}+3$.
    \item[(ii)] If $d=2\ell+1\geq 3$, then $n_{2}+n_{3}=3\ell-2$ and $n_{3}+3n_{4}=\ell^{2}+\ell+2$.
\end{itemize}
\end{proposition}
\begin{proof}
In the case $(i)$ our na\"ive combinatorial count \eqref{naive1} has the form
$$2\ell^{2} + 3\ell = n_{2} + 3n_{3} + 6n_{4}.$$
From Theorem \ref{char} $(i)$, we get $n_{2}+2n_{3}+3n_{4} = \ell^{2}+3\ell-3$, and after combining it with the above combinatorial count we arrive at $n_{3} + 3n_{4} = \ell^{2}+3$. Using again the na\"ive combinatorial count, we obtain
$$2\ell^{2} + 3\ell = n_{2} + 3n_{3} + 6n_{4} = n_{2} + n_{3} + 2(n_{3}+3n_{4}) = n_{2}+n_{3} + 2\ell^{2}+6,$$
and hence $n_{2}+n_{3}=3\ell-6$, which gives us the desired formula. We proceed analogously for the case $(ii)$. 
\end{proof}
\begin{remark}
The above result indeed allows us to find all weak-combinatorial types of $\mathscr{M}$-arrangements with just one conic. For instance, if $d=6$, then the following system of equations must be solved:
$$\begin{cases}
n_{2}+n_{3}=3\\
n_{3}+3n_{4}=12
\end{cases}.$$
There are exactly two solutions in $\mathbb{Z}_{\geq 0}^3$, namely $(n_{2},n_{3},n_{4}) \in \{(3,0,4),(0,3,3)\}$.
\end{remark}
Next, we focus on possible obstructions to the number of singular points of $\mathcal{CL}$ situated on a smooth conic $C \in \mathcal{CL}$. 
\begin{proposition}
\label{mvp}

Let $\mathcal{CL} \subset \mathbb{P}^{2}_\mathbb{C}$ be an $\mathscr{M}$-arrangement of $d\geq 3$ lines and one conic $C$ having only ordinary singularities of multiplicity at most $4$. Denote by $\mathcal{CL}' = \mathcal{CL}\setminus \{C\}$ and let $r = |\mathcal{CL}' \cap C|$. 
\begin{enumerate}
\item[\textbf{a)}] Assume that $d = 2\ell \geq 4$. 
\begin{itemize}
    \item[$(\bullet)$] If $r=2m$, then the only possible values for $m$ are $m= \ell+2$ or $m \leq \ell-1$.
    \item[$(\bullet \bullet)$] If $r=2m+1$, then $m \leq \ell-2$.
\end{itemize}
\item[\textbf{b)}] Assume that $d = 2\ell + 1 \geq 3$. 
\begin{itemize}
    \item[$(\bullet)$] If $r=2m$, then the only possible values for $m$ are $m=\ell+2$ or $m \leq \ell$.
    \item[$(\bullet \bullet)$] If $r=2m+1$, then $m \leq \ell-1$.
\end{itemize}
\end{enumerate}
\end{proposition}
\begin{proof}
 We are going to use \cite[Theorem 1.4 (i)]{Macinic}. Recall that if $\mathcal{CL}$ is an $\mathscr{M}$-arrangement with $d\geq 3$ lines and one conic $C$, then either
 \begin{enumerate}
\item $d = 2\ell$ and $\mathcal{CL}$ is free with the exponents $(\ell-1, \ell+2)$, or
\item $d = 2\ell +1$ and $\mathcal{CL}$ is free with the exponents $(\ell, \ell+2)$.
 \end{enumerate}
 With the above data in hand, we can apply the aforementioned result directly.
\end{proof}
\begin{remark}
It is worth emphasizing directly that the number of intersection points $r = |\mathcal{CL}' \cap C|$ plays a crucial role and we are going to explain it now using the so-called addition technique. Assume that $\mathcal{L}$ is an $\mathscr{M}$-arrangement of $d\geq 3$ lines that is not a pencil and let $C$ be a conic such that $2d = |\mathcal{L}\cap C|$, the maximal possible number according to B\'ezout's Theorem. Let $\mathcal{CL} := \mathcal{L}\cup C = \{f=0\}$ and denote by $d_1\leq  d_{2}$ the first two smallest degrees of a minimal system of generators for the module of Jacobian syzygies ${\rm AR}(f)$. Now, according to \cite[Theorem 1.6]{ADP}, we have  $d_{1}+d_{2} = {\rm deg}(\mathcal{CL})+1$ and hence $\mathcal{CL}$ is \textbf{never} free.
\end{remark}
\section[Combinatorial Poincar\'e polynomials for M-arrangements of conics and lines]{Combinatorial Poincar\'e polynomials for $\mathscr{M}$-arrangements of conics and lines}
In this section we elaborate on combinatorial Poincar\'e polynomials for $\mathscr{M}$-arrangements of conics and lines with ordinary singularities. Let us recall the following definition.
\begin{definition}[{\cite[Definition 1.4]{Pok}}]
Let $\mathcal{CL} \subset \mathbb{P}^{2}_{\mathbb{C}}$ be an arrangement of $k$ conics and $d$ lines such that it admits only ordinary intersection points. Then the combinatorial Poincar\'e polynomial of $\mathcal{CL}$ is defined as
$$\mathfrak{P}(\mathcal{CL};t) = 1+(2k+d-1)t + \bigg(\sum_{r\geq 2}(r-1)n_{r}-d+1 \bigg)t^{2}.$$
\end{definition}
It was proved by the second author that if $\mathcal{CL}$ is a free arrangement of $d$ lines and $k$ conics with only \textit{ordinary quasi-homogeneous singularities} with the exponents $(d_{1},d_{2})$, then $\mathfrak{P}(\mathcal{CL};t)$ splits over the rationals, and we have 
\begin{equation}
    \mathfrak{P}(\mathcal{CL};t) = (1+d_{1}t)(1+d_{2}t),
\end{equation}
see \cite[Theorem 1.5]{Pok}.
\begin{remark}
Observe that if $k=0$, i.e., our arrangement consists of nothing but lines, then the combinatorial Poincar\'e polynomial coincides with the well-known Poincar\'e polynomial for line arrangements.
\end{remark}
As at the end of the previous section, we concentrate on the case of $\mathscr{M}$-arrangements $\mathcal{CL}\subset \mathbb{P}^{2}_{\mathbb{C}}$ of $d$ lines and one conic $C$ with ordinary quasi-homogeneous singularities. The result below presents an explicit formula for the Poincar\'e polynomial of the line arrangement $\mathcal{L} = \mathcal{CL} \setminus \{C\}$.
\begin{theorem}
\label{poil}
Let $\mathcal{CL}\subset \mathbb{P}^{2}_{\mathbb{C}}$ be an $\mathscr{M}$-arrangement of $d\geq 3$ lines and one conic with only ordinary singularities of multiplicity at most $4$. Consider the deletion arrangement $\mathcal{L} = \mathcal{CL}\setminus \{C\}$ and let $r = |C \cap \mathcal{L}|$ --- we assume that $r$ is finite.
\begin{enumerate}
    \item If $d=2\ell\geq 4$, then $\mathfrak{P}(\mathcal{L};t) = (\ell^{2}+\ell-2-r)t^{2}+(2\ell-1)t+1.$
    \item If $d=2\ell+1\geq 3$, then $\mathfrak{P}(\mathcal{L};t) = (\ell^{2}+2\ell-r)t^{2}+2\ell t+1.$
\end{enumerate}
\end{theorem}
\begin{proof}
We will use a general result devoted to Poincar\'e polynomials \cite[Theorem 1.2]{Pok1}, which can be formulated as follows. If $C_{1}\cup C_{2}$ is a reduced plane curve admitting only quasi-homogeneous singularities and such that $|C_{1} \cap C_{2}| = r$ is finite, then
$$\mathfrak{P}(C_{1} \cup C_{2};t) = \mathfrak{P}(C_{1};t) + \mathfrak{P}(C_{2};t) + t-1 + (r-1)t^{2},$$
whereby, for a reduced plane curve $C$ of degree $e$, its Poincar\'e polynomial is defined as
$$\mathfrak{P}(C;t) := 1+(e-1)t + ((e-1)^{2} - \tau(C))t^{2}.$$
Recall that if $C$ is a smooth conic then $\mathfrak{P}(C;t) = 1+t+t^{2}$, and hence we get
\begin{equation}
\label{pp}
\mathfrak{P}(\mathcal{CL};t) = \mathfrak{P}(\mathcal{L};t) + 2t + rt^{2}.
\end{equation}
\begin{enumerate}
\item Assume that $d = 2\ell \geq 4$. Since $\mathcal{CL}$ is an $\mathscr{M}$-arrangement we have the following identity:
$$\mathfrak{P}(\mathcal{CL};t) = (1+d_{1}t)(1+d_{2}t) = (1+(\ell-1)\cdot t)(1+(\ell+2)\cdot t) = (\ell^{2}+\ell-2)t^{2}+(2\ell+1)t+1.$$
Plugging this into \eqref{pp}, we obtain
$$\mathfrak{P}(\mathcal{L};t) = (\ell^{2}+\ell-2-r)t^{2}+(2\ell-1)t+1.$$
\item Assume that $d = 2\ell +1 \geq 3$. We use again the fact that $\mathcal{CL}$ is an $\mathscr{M}$-arrangement obtaining
$$\mathfrak{P}(\mathcal{CL};t) = (1+d_{1}t)(1+d_{2}t) = (1+\ell\cdot t)(1+(\ell+2)\cdot t) = (\ell^2 + 2\ell)t^{2} + (2\ell+2)t+1,$$
hence
$$\mathfrak{P}(\mathcal{L};t) = (\ell^{2}+2\ell-r)t^{2}+2\ell t+1.$$
\end{enumerate}
This completes the proof.
\end{proof}
\begin{remark}
By Proposition \ref{mvp}, we know that if $\mathcal{CL}\subset \mathbb{P}^{2}_{\mathbb{C}}$ is an arrangement of $d=2\ell\geq 4$ lines and one conic $C$ such that $r = |C \cap \mathcal{L}| = 2m$, where $\mathcal{L} = \mathcal{CL}\setminus \{C\}$, then either $m=\ell+2$ or $m\leq \ell-1$. Assume that $r=2\ell+4$, then by Theorem \ref{poil} we have
$$\mathfrak{P}(\mathcal{L};t) = (\ell^{2}-\ell-6)t^{2}+(2\ell-1)t+1 = (1+(\ell-3)t)(1+(\ell+2)t),$$
which suggests that $\mathcal{L}$ might be a free arrangement with exponents $(\ell-3, \ell+2)$. On the other hand, if $\mathcal{L}$ is an arrangement of $d=2\ell$ lines with ordinary points of multiplicity at most $4$, then by \cite[Theorem 2.1]{Dimca2} we have
$${\rm mdr}(\mathcal{L}) \geq \frac{1}{2} \cdot {\rm deg}(\mathcal{L}) - 2 \geq \ell - 2,$$
a contradiction.
\end{remark}
The example below demonstrates that the bound $m\leq \ell-1$ in Proposition \ref{mvp} $\textbf{a)}(\bullet)$ is sharp, and the resulting arrangement $\mathcal{L}$ in this case might be free.
\begin{example}[{\cite[p. 8]{PP2024}}]
Let us consider the arrangement $\mathcal{CL}_1 \subset \mathbb{P}^{2}_{\mathbb{C}}$ given by the following defining polynomial
$$f(x,y,z) = xy(y+x-z)(y-x-z)(y+x+z)(y-x+z)(x^2 + y^2 - z^2).$$ 
We see that this arrangement has weak combinatorics $C(6,1;3,0,4)$.
Figure \ref{f1} presents a geometric realization of the mentioned arrangement.
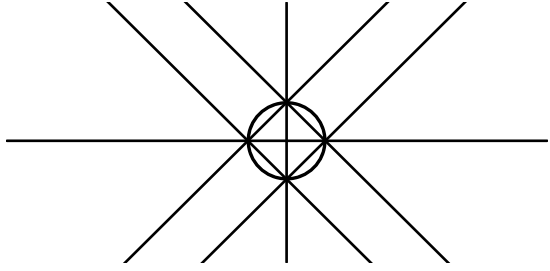
\begin{figure}[h]
\definecolor{yqyqyq}{rgb}{0.5019607843137255,0.5019607843137255,0.5019607843137255}
\definecolor{wqwqwq}{rgb}{0.3764705882352941,0.3764705882352941,0.3764705882352941}
\centering
\begin{tikzpicture}[line cap=round,line join=round,>=triangle 45,x=1.0cm,y=1.0cm,scale=0.5]
\clip(-7.361199998992579,-3.2212539171466634) rectangle (6.856141156665902,3.6615232449100428);
\draw [line width=1.pt] (0.,0.) circle (1.cm);
\draw [line width=1.pt] (0.,0.) circle (1.0177943973136974cm);
\draw [line width=1.pt,domain=-7.361199998992579:6.856141156665902] plot(\x,{(--1.0359054352031525--1.0177943973136971*\x)/1.0177943973136974});
\draw [line width=1.pt,domain=-7.361199998992579:6.856141156665902] plot(\x,{(--1.0359054352031525-1.0177943973136974*\x)/1.0177943973136974});
\draw [line width=1.pt,domain=-7.361199998992579:6.856141156665902] plot(\x,{(--1.0359054352031525-1.0177943973136974*\x)/-1.0177943973136974});
\draw [line width=1.pt,domain=-7.361199998992579:6.856141156665902] plot(\x,{(--1.0359054352031525--1.0177943973136976*\x)/-1.0177943973136974});
\draw [line width=1.pt,domain=-7.361199998992579:6.856141156665902] plot(\x,{(-0.-0.*\x)/2.0355887946273947});
\draw [line width=1.pt] (0.,-3.2212539171466634) -- (0.,3.6615232449100428);
\end{tikzpicture}
\caption{Arrangement $\mathcal{CL}_1$.}
\label{f1}
\end{figure}

\noindent
We see that $r = |C \cap \mathcal{L}|=4$. Using Theorem \ref{poil}, we get
$$\mathfrak{P}(\mathcal{L};t) = (\ell^{2}+\ell-2-r)t^{2}+(2\ell-1)t+1 = 6t^2+5t+1 = (1+2t)(1+3t),$$
which suggests that $\mathcal{L}$ might be free.  This is indeed the case and can be verified through a simple inspection.
\end{example}

\begin{example}[{\cite[p. 9]{PP2024}}]
Now we look at the odd-degree case where $d=2\ell+1$ for some~$\ell$. Consider the arrangement $\mathcal{CL}_{2} \subset \mathbb{P}^{2}_{\mathbb{C}}$ given by
    $$f(x,y,z) = x(x-z)(x+z)(y-z)(y+z)(y-x)(y+x)(x^2 + y^2 -2z^2 ).$$
We see that this arrangement has weak combinatorics $C(7,1;5,2,4)$. Figure \ref{f2} presents a geometric realization of our arrangement.

\begin{figure}[h]
\definecolor{wqwqwq}{rgb}{0.3764705882352941,0.3764705882352941,0.3764705882352941}
\centering
\begin{tikzpicture}[line cap=round,line join=round,>=triangle 45,x=1.0cm,y=1.0cm,scale=0.5]
\clip(-8.626789898750246,-1.5957639615131654) rectangle (8.312131702183478,7.148303929753834);
\draw [line width=1.pt] (0.,-1.5957639615131654) -- (0.,7.148303929753834);
\draw [rotate around={0.:(0.,0.6)},line width=1.pt] (0.,0.6) ellipse (2.065591117977289cm and 1.6cm);
\draw [line width=1.pt,domain=-8.626789898750246:8.312131702183478] plot(\x,{(--2.2--0.6*\x)/1.});
\draw [line width=1.pt,domain=-8.626789898750246:8.312131702183478] plot(\x,{(--2.2-0.6*\x)/1.});
\draw [line width=1.pt,domain=-8.626789898750246:8.312131702183478] plot(\x,{(-1.-1.*\x)/1.});
\draw [line width=1.pt,domain=-8.626789898750246:8.312131702183478] plot(\x,{(-1.--1.*\x)/1.});
\draw [line width=1.pt,domain=-8.626789898750246:8.312131702183478] plot(\x,{(-0.2--0.6*\x)/1.});
\draw [line width=1.pt,domain=-8.626789898750246:8.312131702183478] plot(\x,{(-0.2-0.6*\x)/1.});
\end{tikzpicture}
\caption{Arrangement $\mathcal{CL}_{2}$.}
\label{f2}
\end{figure}
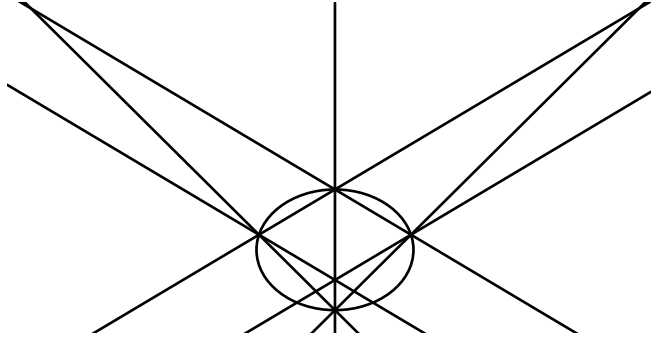

\noindent
Observe that $r = |C \cap \mathcal{L}|=6$, which shows directly that Proposition \ref{mvp} \textbf{b}$(\bullet)$, which states that $m\leq \ell$ is sharp. Using Theorem \ref{poil}, we get
$$\mathfrak{P}(\mathcal{L};t) = (\ell^{2}+2\ell-r)t^{2}+2\ell t+1 = 9t^{2}+6t+1 = (1+3t)(1+3t),$$
which suggests that $\mathcal{L}$ might be free, which is indeed the case.
\end{example}
The next example shows, as one might expect, that the deletion of a smooth conic from an $\mathscr{M}$-arrangement can lead to a non-free line arrangement $\mathcal{L}$.
\begin{example}
Consider the following conic-line arrangement $\mathcal{CL}_{3} \subset \mathbb{P}^{2}_{\mathbb{C}}$ given by
\begin{multline*}
f(x,y,z) = xy(x-z)(x+z)(y+z)(y-z)(y-x-z)(y-x+z)(y-x) \\ (-x^{2}+xy-y^{2}+z^{2}).
\end{multline*}
The arrangement has weak combinatorics $C(9,1;6,4,6)$. Using this information we can compute that $\tau(\mathcal{CL}_{3})=76$ and it gives us that $\mathcal{CL}_{3}$ is an $\mathscr{M}$-arrangement with exponents $(4,6)$.
Figure \ref{f3} presents a geometric realization of the arrangement.

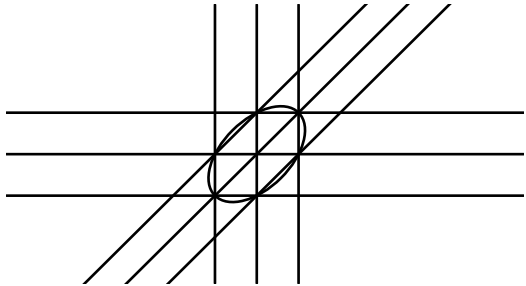
\begin{figure}[h]
\definecolor{uuuuuu}{rgb}{0.26666666666666666,0.26666666666666666,0.26666666666666666}
\centering
\begin{tikzpicture}[line cap=round,line join=round,>=triangle 45,x=1.0cm,y=1.0cm,scale=0.55]
\clip(-5.993665295370997,-3.112645222782914) rectangle (6.387229328409073,3.602513463852773);
\draw [line width=1.pt] (0.,-3.112645222782914) -- (0.,3.602513463852773);
\draw [line width=1.pt] (1.,-3.112645222782914) -- (1.,3.602513463852773);
\draw [line width=1.pt] (-1.,-3.112645222782914) -- (-1.,3.602513463852773);
\draw [line width=1.pt,domain=-5.993665295370997:6.387229328409073] plot(\x,{(-0.-0.*\x)/1.});
\draw [line width=1.pt,domain=-5.993665295370997:6.387229328409073] plot(\x,{(--1.-0.*\x)/1.});
\draw [line width=1.pt,domain=-5.993665295370997:6.387229328409073] plot(\x,{(-1.-0.*\x)/1.});
\draw [line width=1.pt,domain=-5.993665295370997:6.387229328409073] plot(\x,{(--1.--1.*\x)/1.});
\draw [line width=1.pt,domain=-5.993665295370997:6.387229328409073] plot(\x,{(-0.--1.*\x)/1.});
\draw [line width=1.pt,domain=-5.993665295370997:6.387229328409073] plot(\x,{(-1.--1.*\x)/1.});
\draw [rotate around={-135.:(0.,0.)},line width=1.pt] (0.,0.) ellipse (1.4142135623730951cm and 0.816496580927726cm);

\end{tikzpicture}
\caption{Arrangement $\mathcal{CL}_{3}$.}
\label{f3}
\end{figure}
Let $\mathcal{L}=\mathcal{CL}_{3} \setminus \{C\}$, so we end up with an arrangement of $9$ lines such that $n_{2}^{\mathcal{L}} = 6$ and $n_{3}^{\mathcal{L}} = 10$. Observe that this line arrangement cannot be free since its Poincar\'e polynomial has the form $\mathfrak{P}(\mathcal{L};t)=18t^2+8t+1$ and hence it does not split over the rationals. However, $\mathcal{L}$ is \textbf{plus-one generated} with exponents $(4,5,6)$, see \cite[Theorem 3.2 (1)(b)(i)]{Mac}.
\end{example}
The following example gives a new construction of an \(\mathscr{M}\)-arrangement consisting of one conic and \(11\) lines which, to the best of our knowledge, appears here for the first time.
\begin{example}
\label{new}
Consider the following conic-line arrangement $\mathcal{CL}_4 \subset \mathbb P^2_{\mathbb C}$
given by

\begin{multline*}
f(x,y,z)={}
xyz(x-z)(x+z)(y+z)(y-z)(y-x-z)(y-x+z)(y-x)(x+y)\\
(-x^2+xy-y^2+z^2).
\end{multline*}
The arrangement has weak combinatorics $C(11,1;11,2,10)$. Using this information we can compute
\[
\tau(\mathcal{CL}_4)=11+4\cdot 2+9\cdot 10=109.
\]
Moreover, a direct \texttt{SINGULAR} computation gives $\operatorname{mdr}(f)=5$, and hence $\mathcal{CL}_4$
is an $\mathscr{M}$-arrangement with exponents $(5,7)$.

Let $\mathcal{L}=\mathcal{CL}_4\setminus\{C\}$, where
\[
C:\ -x^2+xy-y^2+z^2=0.
\]
Then we end up with an arrangement of $11$ lines such that
\[
n_2^{\mathcal{L}}=7,\qquad n_3^{\mathcal{L}}=8,\qquad n_4^{\mathcal{L}}=4.
\]
The Poincar\'e polynomial of $\mathcal{L}$ has the form
\[
\mathfrak{P}(\mathcal{L};t)
=
1+10t+\bigl(7+2\cdot 8+3\cdot 4-11+1\bigr)t^2
=
25t^2+10t+1
=
(1+5t)^2.
\]
This suggests that the deletion arrangement $\mathcal{L}$ might be free. This is indeed the case: direct computations give $\tau(\mathcal{L})=75$ and $\operatorname{mdr}(L)=5$, therefore $\mathcal{L}$ is free with
exponents $(5,5)$.
\end{example}
\section[Boundedness of real M-arrangements of one conic and lines]{Boundedness of real $\mathscr{M}$-arrangements of one conic and lines}
We now prove a boundedness result for arrangements contained in the real projective plane.
\begin{theorem}
Let $\mathcal{CL}= \mathcal{L}\cup C\subset \mathbb P^2_{\mathbb R}$ be a real $\mathscr{M}$-arrangement
consisting of $d\geq 3$ lines and one smooth conic $C$. Assume that $\mathcal{CL}$ has only
ordinary singularities of multiplicity at most $4$. 
\begin{itemize}
    \item If $d=2\ell$, then $\ell\leq 9$. In particular, there are no such arrangements
    for even $d\geq 20$.
    \item If $d=2\ell+1$, then $\ell\leq 10$. In particular, there are no such arrangements
    for odd $d\geq 23$.
\end{itemize}
\end{theorem}

\begin{proof}
Let $n_i$ denote the number of $i$-fold points of $\mathcal{CL}$, and let
$n_i^{\mathcal{L}}$ denote the number of $i$-fold points of the line arrangement
$\mathcal{L}=\mathcal{CL}\setminus\{C\}$. Put
\[
b_i=\#\{p\in C\cap \mathcal{L}:\ p \text{ is an } i\text{-fold point of } \mathcal{CL}\}.
\]
Then
\begin{equation}
\label{c1}
b_2+2b_3+3b_4=2d. 
\end{equation}
After deleting the conic, points on $C$ drop their multiplicity by one, whereas
points outside $C$ keep their multiplicity. Hence
\begin{equation}
\label{c2}
n_2^{\mathcal{L}}=n_2-b_2+b_3,\qquad
n_3^{\mathcal{L}}=n_3-b_3+b_4,\qquad
n_4^{\mathcal{L}}=n_4-b_4. 
\end{equation}
We can assume that $d\geq 5$ to avoid triviality. Since $\mathcal{L}$ is a real line arrangement with points of multiplicity at most $4$,
Melchior's inequality \cite{Melchior} gives
\[
n_2^{\mathcal{L}}\geq 3+n_4^{\mathcal{L}}.
\]
Using \eqref{c2}, we obtain
\[
n_2-b_2+b_3\geq 3+n_4-b_4.
\]
Together with \eqref{c1}, this gives
\begin{equation}
\label{c3}
3b_3+4b_4\geq 2d+3+n_4-n_2.
\end{equation}

Assume first that $d=2\ell$. By Proposition \ref{kon} (i), we have
\[
n_2+n_3=3\ell-6,\qquad n_3+3n_4=\ell^2+3.
\]
Thus \eqref{c3} gives
\begin{equation}
\label{c4}
3b_3+4b_4
\geq
\frac{\ell^2+3\ell+30+2n_3}{3}.
\end{equation}
On the other hand, from \eqref{c1} and $b_3\leq n_3$ we have
\begin{equation}
\label{c5}
3b_3+4b_4
=
\frac43(2b_3+3b_4)+\frac13b_3
\leq
\frac{16\ell+n_3}{3}. 
\end{equation}
Combining \eqref{c4} and \eqref{c5}, we obtain
\[
\ell^2-13\ell+30+n_3\leq 0.
\]
Hence $\ell\leq 10$. If $\ell=10$, then $n_3=0$, but
\[
n_3+3n_4=\ell^2+3=103
\]
is impossible mod $3$. Therefore $\ell\leq 9$.

Now assume that $d=2\ell+1$. Then by Proposition \ref{kon} (ii) we have
\[
n_2+n_3=3\ell-2,\qquad n_3+3n_4=\ell^2+\ell+2.
\]
The same argument gives
\[
3b_3+4b_4
\geq
\frac{\ell^2+4\ell+23+2n_3}{3}
\]
and
\[
3b_3+4b_4
\leq
\frac{16\ell+8+n_3}{3}.
\]
Therefore
\[
\ell^2-12\ell+15+n_3\leq 0.
\]
Since $n_3\geq 0$, we obtain $\ell\leq 10$.
\end{proof}

\section[A remark on Castelnuovo-Mumford regularity of M-curves]{A remark on the Castelnuovo-Mumford regularity of $\mathscr{M}$-curves}
We conclude the paper by recording precise values of the Castelnuovo-Mumford regularity of the $S$-modules ${\rm AR}(f)$ and $M(f)$ that are associated with $\mathscr{M}$-curves, and this observation should be viewed as a small contribution to the broader picture presented in \cite{CMD}. Recall that to any graded $S$-module $M$ we can associate a coherent sheaf $\widetilde{M}$ on $\mathbb{P}^{2}_{\mathbb{C}}$, the sheafification of $M$.
\begin{definition}
We say that a coherent sheaf $\widetilde{M}$ is $m$-regular if
$$H^{1}(\mathbb{P}^{2}_{\mathbb{C}}, \widetilde{M}(m-1)) = H^{2}(\mathbb{P}^{2}_{\mathbb{C}}, \widetilde{M}(m-2))=0.$$
The minimal $m$ such that $\widetilde{M}$ is $m$-regular is called the Castelnuovo-Mumford regularity of $\widetilde{M}$, and it is denoted by ${\rm reg}(\widetilde{M})$. We also set ${\rm reg}(M) = {\rm reg}(\widetilde{M})$.
\end{definition}

In order to present the main observation of this section, we recall some necessary notions, following the lines of \cite{CMD}. Let $C = \{f=0\} \subset \mathbb{P}^{2}_{\mathbb{C}}$ be a reduced plane curve. Consider the following submodule of Koszul-type relations ${\rm KR}(f)$, i.e., this is the submodule of ${\rm AR}(f)$ generated by the following obvious relations of degree $d-1$, namely
$$(\partial_{y}f, -\partial_{x}f,0), \quad (\partial_{z}f, 0, -\partial_{x}f), \quad (0, \partial_{z}f, -\partial_{y}f).$$
Define $$ER(f) = {\rm AR}(f) / {\rm KR}(f)$$
and 
$${\rm mdr}_{e}(f) = \min \{r \in \mathbb{Z} : {\rm ER}(f)_{r} \neq 0\}.$$
If ${\rm mdr}(f) < d-1$ then ${\rm mdr}_{e}(f) = {\rm mdr}(f)$, and this situation occurs, for instance, if our curves are free.
\begin{definition}
For a reduced polynomial $f \in S_{d}$ we define
\begin{enumerate}
    \item the coincidence threshold as
    $${\rm ct}(f) = \max \{q : {\rm dim} \, M(f)_{k} = {\rm dim} \, M(g)_{k} \text{ for all } k\leq q\},$$
    where $g$ is a homogeneous polynomial of $S$ of the same degree $d$ as $f$ having the property that $C' = \{g=0\}\subset \mathbb{P}^{2}_{\mathbb{C}}$ is a smooth curve.
    \item the stability threshold
    $${\rm st}(f) = \min \{q : {\rm dim} \, M(f)_{k} = \tau(C) \text{ for all } k \geq q\}.$$
\end{enumerate}
\end{definition}
\noindent
It turns out that if $C = \{f=0\}\subset \mathbb{P}^{2}_{\mathbb{C}}$ is a reduced singular curve, then
$${\rm ct}(f) = {\rm mdr_{e}}(f) + d-2.$$
Moreover, one has the following relation involving ${\rm ct}(f)$ and ${\rm st}(f)$ in the case where $C = \{f=0\} \subset \mathbb{P}^{2}_{\mathbb{C}}$ is free, namely one has
\begin{equation}
{\rm ct}(f) + {\rm st}(f) = 3(d-2).
\end{equation}
Moreover, we have 
$${\rm st}(f) = 3(d-2) - {\rm ct}(f) = 2d-4-{\rm mdr}(f)=d-3+d_{2},$$
where $d_{2}$ is the second exponent of the free curve $C$.
We will need the following result \cite[Theorem 2.6]{CMD}.
\begin{theorem}
Let $C = \{f=0\}\subset \mathbb{P}^{2}_{\mathbb{C}}$ be a reduced singular curve of degree $d$. Then the equality ${\rm reg}(M(f)) = {\rm st}(f)$ holds if and only if $C$ is free.
\end{theorem}
\noindent
Finally, the following relation also holds \cite[p. 5]{CMD}:
\begin{equation}
{\rm reg}(M(f)) = {\rm reg}({\rm AR}(f))+d-3.
\end{equation}
Now we are ready to formulate our main observation in this section.
\begin{proposition}
Let $C = \{f=0\}\subset \mathbb{P}^{2}_{\mathbb{C}}$ be an $\mathscr{M}$-curve of degree $d$. Then
$${\rm reg}(M(f)) = \begin{cases}
3m-2 & \text{ if } d=2m\geq 6, \text{ or}\\
3m-1 & \text{ if } d=2m+1 \geq 5
\end{cases}$$
and
$${\rm reg}({\rm AR}(f)) = m+1$$
both for $d=2m\geq 6$ or $d=2m+1\geq 5$.
\end{proposition}
\begin{proof}
It is enough to observe that 
$${\rm reg}({\rm AR}(f)) = {\rm st}(f) - d + 3 = d_{2},$$
and the claim follows from the fact that, if $C = \{f=0\} \subset \mathbb{P}^{2}_{\mathbb{C}}$ is an $\mathscr{M}$-curve of degree $d$, then $C = \{f=0\}$ is free with exponents $(d_{1},d_{2})$, where
\begin{align*}
(d_{1},d_{2}) = (m-2, m+1) & \text{ if } d=2m\geq 6, \text{ or} \\
(d_{1},d_{2}) = (m-1, m+1) & \text{ if } d=2m+1\geq 5.
\end{align*}
This gives us ${\rm reg}({\rm AR}(f)) = m+1$, and then the values for ${\rm reg}(M(f))$ follow.
\end{proof}
\section*{Funding}
This paper is based on research conducted during an intensive research group organized by the University of the National Education Commission Krakow in Ogrodzieniec. We would like to thank The Witcher for the inspiration.

Marek Janasz and Piotr Pokora are supported by the National Science Centre (Poland) Sonata Bis Grant  \textbf{2023/50/E/ST1/00025}. For the purpose of Open Access, the authors have applied a CC-BY public copyright license to any Author Accepted Manuscript (AAM) version arising from this submission.

\bigskip
Affiliation of the authors:
\noindent
Department of Mathematics,
University of the National Education Commission Krakow,
Podchor\k{a}\.zych 2,
PL-30-084 Krak\'ow, Poland. \\
\nopagebreak
\noindent
Marek Janasz: \texttt{marek.janasz@uken.krakow.pl} \\
Piotr Pokora: \texttt{piotr.pokora@uken.krakow.pl}

\begin{thebibliography}{9}

\bibitem{ADP}  T. Abe, A. Dimca, and P. Pokora, A new hierarchy for complex plane curves. \textit{Canad. Math. Bull.} Advance Publication 1 -- 24 (2025), \url{https://doi.org/10.4153/S0008439525101422}.

\bibitem{BM}
V. Beorchia and R. M. Mir\'o-Roig, Jacobian schemes of conic-line arrangements and eigenschemes.
\textit{Mediterr. J. Math.} \textbf{21(1)}: Paper No. 16, 23 p. (2024).

\bibitem{CC}
R. de Moura Canaan and S. C. Coutinho, On invariant line arrangements. \textit{Discrete Comput. Geom.} \textbf{51}: 337 -- 361 (2014).

\bibitem{RCS}  
A. Dimca,  \textit{Topics on real and complex singularities. An introduction.} Advanced Lectures in Mathematics. Braunschweig/Wiesbaden: Friedr. Vieweg \& Sohn. 1987.

\bibitem{Dimca}
A. Dimca,  \textit{Hyperplane arrangements. An introduction}. Universitext. Cham: Springer. xii, 200 p. (2017).

\bibitem{Dimca2}
A. Dimca, Unexpected curves in $\mathbb{P}^{2}$, line arrangements, and minimal degree of Jacobian relations. \textit{J. Commut. Algebra} \textbf{15(1)}: 15 -- 30 (2023).

\bibitem{DimPok} 
A. Dimca and P. Pokora, Maximizing Curves Viewed as Free Curves. \textit{Int. Math. Res. Not. IMRN} \textbf{22}:  19156 -- 19183 (2023).

\bibitem{CMD}
A. Dimca, On the Castelnuovo-Mumford regularity of curve arrangements. \textit{Rev. Roum. Math. Pures Appl.} \textbf{70(1-2)}: 73 -- 84 (2025).

\bibitem{duP}
A. Du Plessis and C. T. C. Wall, Application of the theory of the discriminant to highly singular plane curves. \textit{Math. Proc. Camb. Philos. Soc.} \textbf{126(2)}: 259 -- 266 (1999).

\bibitem{JanLes} M. Janasz and I. Le\'sniak, On the existence of maximizing curves of odd degrees. \textit{Proc. Am. Math. Soc.} \textbf{153(11)}: 4633 -- 4642 (2025).

\bibitem{JanLes05}
M. Janasz and I. Le\'sniak, On free line arrangements with double, triple and quadruple points. \textit{European Journal of Mathematics} \textbf{11}: Art. Id. 81 (2025).

\bibitem{Mac} 
A. M\u{a}cinic, Deletion--addition of a smooth conic for free curves. \textit{Nagoya Math. J.} \textbf{261}: Paper No. e20, 19 p. (2026).

\bibitem{Macinic}
A. M\u{a}cinic and J. Vall\'es, Addition-deletion for conic-line arrangements with split Chern polynomial. \textbf{arXiv:2510.02771}.

\bibitem{Melchior}  
E. Melchior, \"{U}ber Vielseite der Projektive Ebene. Deutsche Mathematik {\bf 5}: 461 – 475 (1941).

\bibitem{NP26}
B. Naskręcki and P. Pokora, On the geography of log-surfaces. \textit{EMS Surv. Math. Sci.} (2026), published online first. \url{ https://doi.org/10.4171/EMSS/113}.
\bibitem{neusel}
M. Neusel, \textit{Konfigurationen von Geraden und einer Quadrik mit Baumaufl\"osung – ein Bilderbuch.} G\"ottingen: Univ. G\"ottingen, Math.-Naturwiss. Fachber. 132 pp. (1992).


\bibitem{Persson}
U. Persson, Horikawa surfaces with maximal Picard numbers. \textit{Math. Ann.} \textbf{259}: 287 -- 312 (1982).

\bibitem{Pard}
R. Pardini and P. Pokora, Bounds for characteristic numbers of conic-line arrangements in the plane. \textit{Collect. Math.} Advance Publication 1 -- 12 (2025), \url{https://doi.org/10.1007/s13348-025-00492-w}.

\bibitem{Pelka}
T. Pe\l ka, Smooth $\mathbb{Q}$-homology planes satisfying the negativity conjecture. \textit{J. Lond. Math. Soc., II. Ser.} \textbf{108(3)}: 1193 -- 1274 (2023).

\bibitem{PP2024}
P. Pokora, Freeness of arrangements of lines and one conic with ordinary quasi-homogeneous singularities. \textit{Taiwanese Math. Journal}. \textbf{29(6)}: 1651 -- 1666 (2025).

\bibitem{Pok1}
P. Pokora, On the addition technique for Betti and Poincar\'e polynomials of plane curves. \textit{Bulletin Polish Acad. Sci. Math.} \textbf{72(2)}: 111 -- 118 (2024).

\bibitem{Pok}
P. Pokora, On Poincar\'e polynomials for plane curves with quasi-homogeneous singularities. \textit{Bull. Lond. Math. Soc.} \textbf{57(8)}: 2549 -- 2560 (2025).

\bibitem{Reiffen} 
H. Reiffen, Das Lemma von Poincar\'e f\"ur holomorphe Differentialformen auf komplexen R\"aumen. \textit{Math. Z.} \textbf{101}:  269 -- 284 (1967). 

\bibitem{ST}
H. Schenck and \c{S}. Toh\v{a}neanu, Freeness of conic-line arrangements in $\mathbb{P}^{2}$. \textit{Comment. Math. Helv.} \textbf{84(2)}: 235 -- 258 (2009).

\end{thebibliography}
\end{document}